\documentclass[a4paper,3p]{elsarticle}
\usepackage{times}
\usepackage{amsmath,amsfonts,amssymb,latexsym}
\usepackage{cleveref}
\usepackage{subcaption}
\usepackage{graphicx}
\usepackage{pgfplots}
\usepackage{booktabs}
\usepackage{array}
\usepackage{fancyvrb}
\usepackage{listings}
\pgfplotsset{compat=newest,
width=\linewidth,
height=\linewidth,
ylabel shift={-3pt},
xlabel shift={-3pt}
}
\usepackage{color}
\newtheorem{theorem}{Theorem}

\newenvironment{proof}[1][Proof]{\textbf{#1.} }{\ \rule{0.5em}{0.5em}}

\newcommand{\R}{\mathbb{R}}

\newcommand{\e}{{\rm e}}

\begin{document}

\begin{frontmatter}
\title{D-splitting methods: 2$N$-storage embedded explicit pseudo-geometric Runge-Kutta methods using splitting methods}

%\author{Sergio Blanes}\ead{serblaza@imm.upv.es}
%\address{Instituto Universitario de Matem\'{a}tica Multidisciplinar, %\\
%Universitat Polit\`{e}cnica de Val\`{e}ncia. Spain}

%Departament de Matem\`atiques and IMAC, Universitat Jaume I, 12071-Castell\'on, Spain
%email: alescori@uji.es

\author[imm]{Sergio Blanes}\ead{serblaza@imm.upv.es}
\address[imm]{Instituto Universitario de Matem\'{a}tica Multidisciplinar, %\\
Universitat Polit\`{e}cnica de Val\`{e}ncia. Spain}

\author[uji]{Alejandro Escorihuela-Tom\`as}\ead{alescori@uji.es}
\address[uji]{Departament de Matem\`atiques and IMAC, Universitat Jaume I, 12071-Castell\'on, Spain}

\begin{abstract}
	Low-storage explicit Runge–Kutta
	schemes are particularly popular for the numerical integration of time-dependent partial differential equations based on the method-of-lines due to their efficiency and their reduced memory requirements. We show that D-splitting methods, splitting methods on the duplicated phase space, can be used as high performance 2N-storage embedded explicit RK methods without a third storage register. They are pseudo-geometric methods preserving some of the qualitative properties of the exact solution up to a higher order than the order of the method. Some of their properties are analysed and new tailored methods are built whose high performance %new methods are proposed and 
	is tested on numerical examples.
	%, and we indicate how to build new tailored methods.
\end{abstract}
\end{frontmatter}

%\date{}
%\maketitle

\section{Introduction} \label{Introduction}

The numerical simulation of many dynamical systems {requires} the numerical integration of high-dimensional systems of differential equations. This is the case for example, in Computational Fluid Dynamics or in Quantum Mechanics where coupled systems of PDEs in several dimensions must be solved.
To achieve high accuracy, precise spatial discretization is generally required, which can involve up to millions of mesh points. If, for example, the method of lines is used to advance in time, one has to numerically solve a system of up to millions of coupled ODEs
%To reach high accuracy usually needs an accurate spatial discretisation which can involve up to millions of grid points. If, for example, the method of lines is used to advance in time, a system of up to millions of coupled ODEs needs to be numerically solved
\begin{equation}\label{eq:ODE1}
	x^{\prime}= f(x),
	\qquad
	x(t_0)=x_{0}\in\mathbb{C}^N,
\end{equation}
with $'\equiv \frac{d}{dt}$ and $N\gg 1$ (to simplify the presentation, we consider the autonomous problem and the extension to the non-autonomous case is left at the conclusions).

 First- and second-order methods in time are the most frequently used methods. However, when a high accuracy is  required (achieving at least similar accuracy in time integration to space integration), higher-order methods for solving ODEs are generally more efficient.
 The performance of a method for solving ODEs is usually measured by the accuracy achieved relative to its computational cost. For non-stiff problems, explicit methods are often used instead of implicit ones, as they are faster and simpler to implement in a code. 
 The computational cost is usually measured as the number of evaluations of the vector field, $f(x)$, at each step (or any quantity proportional to this number). However, for very high-dimensional problems, storage requirements are significant, so a combination of high accuracy and low memory consumption is sought.
For this reason, explicit low-storage Runge-Kutta (RK) methods have been developed \cite{bazavov252sr,calvo04anm,carpenter94fo2,higueras25lse,ketcheson10rkm,kennedy00lse,williamson80lsr} and frequently used in the literature \cite{calvo04anm,kennedy00lse}. Methods up to the fifth order have been obtained showing a good performance. However, they are obtained by solving a relatively complex set of order conditions, which makes constructing high order methods and analyzing them a difficult task.
  
 In this {paper}, we present a novel procedure to build, in a very simple way, 2$N$-storage RK methods at any order. They correspond to splitting methods for general separable systems in a duplicated phase space, D-splitting methods, which is very simple to implement. There are many highly efficient splitting methods in the literature at any order, where some of them appear on libraries \cite{moayeri26tat}, which can be used as partly optimised 2$N$-storage RK methods, and new tailored ones can also be built. 
 
 Frequently, in practice, the vector field has an algebraic structure that imposes to the solution some geometric properties which are very important to be preserved,  like in classical or quantum Hamiltonian systems. In those cases, geometric integrators have shown to be highly efficient \cite{blanes25aci,hairer06gni,iserles08afc,sanz-serna94nhp}. We prove in the linear case that the new D-splitting methods are pseudo-Lie group integrators, preserving the Lie-group structure up to a higher order than the order of the method as already shown in \cite{mcLachlan25rkm} for classical Hamiltonian systems so, they can also be used as pseudo-geometric methods. 
 Finally, we present a backward error analysis which allows to build new tailored splitting methods with a reduced number of stages.
 % to reach a given order.
 %, and  new optimised methods are obtained.

\section{Runge-Kutta and 2$N$-storage methods}

An explicit $s$-stage Runge-Kutta method to advance a time step, $h$, can be written as \cite{butcher08nmf,hairer93sod}
{
\begin{eqnarray}
	{k}_{1} &=& f(x_{n}), \nonumber \\
	{k}_{i} &=& f({x}_n + h( a_{i1}k_{1}+\ldots+ a_{i,i-1}k_{i-1})) ,  \qquad i=2,\ldots,s,  %\label{eq.RK1}
	 \nonumber \\
%	{k}_{i} &=& f({x}_n + h\sum_{j=1}^{i-1} a_{ij}k_{j}) , \qquad i=2,\ldots,s, \label{eq.RK}  \label{eq.RK1} \\
	{x}_{n+1} &=& {x}_n +h( b_{1}k_{1}+\ldots+b_{s}k_{s})  , \label{eq.RK2}
%	{x}_{n+1} &=& {x}_n +h\sum_{j=1}^{s} b_{j}k_{j}  , \label{eq.RK2}
\end{eqnarray}
}
with $x_n\simeq x(t_n), \ t_n=t_0+nh$. Embedded methods incorporate another set of coefficients, $\hat b_i$ such that $\hat{x}_{n+1}={x}_n+h\sum_{j=1}^{s} \hat b_{j}k_{j}$ provides a second approximation of lower order to the solution, allowing to estimate the local error to adjust the time step. We say an $s$-stage method is of order $p$ if $x_1=x(h)+{\cal O}(h^{p+1})$ and is denoted as $(s,p)$.
%, where the coefficients $a_{ij},b_i$ have to satisfy a set of order conditions.

One step involves $s$ evaluations of the vector field, but also the $s$ vectors $k_i, \ i=1,2,\ldots,s$ have to be stored in memory since they are required in \eqref{eq.RK1} and \eqref{eq.RK2} when the coefficients $a_{ij}, b_i$ are arbitrary. 
However, it is possible to reduce the number of vectors to be stored in memory by adding appropriate conditions to the coefficients. The minimum number of storage locations per step is given by a class of RK methods introduced by Williamson which are referred as 2$N$-storage methods of Williamson type (W),
%. The algorithm to compute one step is the following 
%\begin{eqnarray}
%	\Delta y_1 &=& hf(x_{n}), \nonumber \\
%	{y}_{1} &=& {x}_{n} + B_1\Delta y_1,  \nonumber  \\
%	\Delta y_i &=& A_i \Delta y_{i-1} + h f(y_{i-1}) , \nonumber\\
%{y}_{i} &=& {y}_{i-1} + B_i\Delta y_i   , \qquad\qquad\qquad i=2,\ldots,s, %\label{eq.2N_RK1}
%\end{eqnarray}
%and $x_{n+1}=y_s$,
which keeps only two vectors in memory and depends on $2s-1$ parameters, $A_i,B_i, \ i=1,2,\ldots,s$ with $A_1=0$. This procedure to build 2$N$-storage methods is not unique, e.g. the scheme in the two-register van der Houwen (vdH) format \cite{houwen77coi}, which are both given by 
 ($y_0=x_n$)
%\begin{eqnarray}
%	k_1 &=& hf(y_{0}), \nonumber \\
%	{y}_{1} &=& {x}_{n} + hb_1 k_1,  \nonumber  \\
%	k_i &=&  h f(y_{i-2}+ha_{i,i-1}k_{i-1}),  \nonumber\\
%	{y}_{i} &=& {y}_{i-1} + hb_i k_i   , \qquad\qquad\qquad i=2,\ldots,s, \label{eq.2N_RK_vdH}
%\end{eqnarray}
\begin{equation}\label{eq.2N_RK_W_vdH}
	(W):\left\{
	\begin{array}{rcl}
		\Delta y_1 &=& hf(y_{0}), \nonumber \\
		{y}_{1} &=& {x}_{n} + B_1\Delta y_1,  \nonumber  \\
		\Delta y_i &=& A_i \Delta y_{i-1} + h f(y_{i-1}) , \nonumber\\
		{y}_{i} &=& {y}_{i-1} + B_i\Delta y_i   , \qquad\quad i=2,\ldots,s, %\label{eq.2N_RK1}
	\end{array}
	\right.
	(vdH):\left\{
	\begin{array}{rcl}
		k_1 &=& f(y_{0}), \nonumber \\
		{y}_{1} &=& {x}_{n} + hb_1 k_1,  \nonumber  \\
		k_i &=&  f(y_{i-2}+ha_{i,i-1}k_{i-1}),  \nonumber\\
		{y}_{i} &=& {y}_{i-1} + hb_i k_i   , \qquad\qquad i=2,\ldots,s, 
	\end{array}
	\right.
\end{equation}
and $x_{n+1}=y_s$. The following pseudo-codes show that these schemes require to store only two registers per step
%($Y=x_n$)\
 \\
%\begin{eqnarray}
%	k_1 &=& hf(y_{0}), \nonumber \\
%	{y}_{1} &=& {x}_{n} + hb_1 k_1,  \nonumber  \\
%	k_i &=&  h f(y_{i-2}+ha_{i,i-1}k_{i-1}),  \nonumber\\
%	{y}_{i} &=& {y}_{i-1} + hb_i k_i   , \qquad\qquad\qquad i=2,\ldots,s, \label{eq.2N_RK_vdH}
%\end{eqnarray}
% \begin{equation}\label{eq.2N_RK_W_vdH}
% 	(W):\left\{
% 	\begin{array}{rcl}
% 		\Delta y_1 &=& hf(y_{0}), \nonumber \\
% 		{y}_{1} &=& {x}_{n} + B_1\Delta y_1,  \nonumber  \\
% 		\Delta y_i &=& A_i \Delta y_{i-1} + h f(y_{i-1}) , \nonumber\\
% 		{y}_{i} &=& {y}_{i-1} + B_i\Delta y_i   , \quad\quad i=2,\ldots,s, %\label{eq.2N_RK1}
% 	\end{array}
% 	\right.
% 	(vdH):\left\{
% 	\begin{array}{rcl}
% 		k_1 &=& hf(y_{0}), \nonumber \\
% 		{y}_{1} &=& {x}_{n} + hb_1 k_1,  \nonumber  \\
% 		k_i &=&  h f(y_{i-2}+ha_{i,i-1}k_{i-1}),  \nonumber\\
% 		{y}_{i} &=& {y}_{i-1} + hb_i k_i   , \quad\qquad i=2,\ldots,s, 
% 	\end{array}
% 	\right.
%       \end{equation}

\begin{minipage}[c]{0.45\textwidth}
\centering
(W)
\begin{lstlisting}
        $\Delta Y$ := $0$
        $Y$ := $x_n$
        for i=1:s do
          $\Delta Y$ := $A_i\Delta Y +hf(Y)$
          $Y$ := $Y +B_i\Delta Y$
        end
        $x_{n+1}$ := $Y$
\end{lstlisting}
\end{minipage}
\hfill
\begin{minipage}[c]{0.45\textwidth}
\centering
(vdH)
\begin{lstlisting}
        $K$ := $f(Y)$
        $Y$:= $x_n + hb_1 K$
        for i=2:s do
          $K$ := $Y+(a_{i,i-1}-b_{i-1})hK$
          $K$ := $f(K)$
          $Y$ := $Y + hb_iK$
        end  
        $x_{n+1}$ := $Y$
\end{lstlisting}
\end{minipage}
%and $x_{n+1}=Y$. 
The vdH schemes is usually referred as 2R methods, which corresponds to the scheme \eqref{eq.RK1}-\eqref{eq.RK2} when $a_{i,j}=b_j, \ i=j+1,j+2,\ldots,s$. Several 2R embedded methods of order three, four and five are given in \cite{kennedy00lse}. 
Embedded methods are provided, but then a third storage register is required, increasing the storage cost.

\section{D-splitting methods: Splitting methods in the duplicated phase space}

%\subsection{Brief review on splitting methods}

Suppose the vector field, $f(x)$, can be decomposed into a sum of two  contributions,   $f(x) =  f^{[1]}(x) + f^{[2]}(x)$, in such a way that each sub-problem 
${x}' = f^{[i]}(x),  \quad i = 1, 2, $
%\[
%{x}' = f^{[i]}(x),  \qquad \qquad i = 1, 2
%\] 
can be integrated exactly 
%(or, more generally, it is simpler to integrate than the original system), 
with solutions $x(h) = \varphi_h^{[i]}(x_0)$ at $t =t_0+ h$. Then, a consistent composition of these flows  provides approximations to the exact solution. %\cite{blanes25aci,mclachlan02sm}. 
For example, a symmetric second-order method is given by
\begin{equation}   \label{eq.2.spt}
	S_h^{[2]} =
	\varphi_{\frac{h}2}^{[1]} \circ \varphi_h^{[2]} \circ	\varphi_{\frac{h}2}^{[1]}, 
	\qquad \mbox{i.e.} \qquad 
	x_1= \varphi_{\frac{h}2}^{[1]} \Big(\varphi_h^{[2]} \big(\varphi_{\frac{h}2}^{[1]}(x_0)\big)\Big).
\end{equation}
{If we define the Lie operators $\hat A^{[i]}\equiv f^{[i]}(x)\cdot \nabla, \ i=1,2,$ then the exact solution is formally given by $x(t_{n}+h)=\phi_h(x(t_n))$, with $\phi_h=\e^{h(\hat A^{[1]}+\hat A^{[2]})}$ and the method can be written as
%\begin{equation}   \label{eq.Strang}
%	S_h^{[2]} =\e^{\frac{h}2\hat A^{[1]}}\e^{h\hat A^{[2]}}\e^{\frac{h}2\hat A^{[1]}} = 
%\e^{h(\hat A^{[1]}+\hat A^{[2]} + h^2 \hat E_2 + h^4\blau{\hat{E}_4}+\cdots)}
%\end{equation} 
%	($\hat E_k$ are error terms that depend on {nested} commutators of the Lie operators $\hat A^{[i]}$) where the error can be expanded in even powers of $h$ and it is time-symmetric.
%Similarly, an $s$-stage splitting method of order $p$ is given by the composition
%\begin{equation}   \label{eq.Splitting}
%  \Psi_h^{[p]} =
%\e^{hb_s \hat A^{[2]}} \e^{ha_s \hat A^{[1]}}
%\cdots
%\e^{hb_1 \hat A^{[2]}} \e^{ha_1\hat A^{[1]}}
% = 
%\exp\left(h(\hat A^{[1]}+\hat A^{[2]}) + h^{p+1} \hat E_{p+1} + {\cal O}(h^{p+2}))\right).
%\end{equation}
{
\begin{equation}   \label{eq.Strang}
	S_h^{[2]} =\e^{\frac{h}2\hat A^{[1]}}\e^{h\hat A^{[2]}}\e^{\frac{h}2\hat A^{[1]}} = 
\e^{h(\hat A^{[1]}+\hat A^{[2]}) + h^3 \hat E_3 + h^5{\hat{E}_5}+\cdots}
\end{equation}
where the error can be expanded in odd powers of $h$, and $\hat{E}_k$ are error terms
that depend on nested commutators of $\hat{A}^{[i]}$.
In general, an $s$-stage splitting method of order $p$ is given by the composition  
\begin{equation}   \label{eq.Splitting}
  \Psi_h^{[p]} =
% \e^{ha_{s+1} \hat A^{[1]}}
\e^{hb_s \hat A^{[2]}} \e^{ha_s \hat A^{[1]}}
\cdots
\e^{hb_1 \hat A^{[2]}} \e^{ha_1\hat A^{[1]}}
 = 
\exp\left(h(\hat A^{[1]}+\hat A^{[2]}) + h^{p+1} \hat E_{p+1} + {\cal O}(h^{p+2})\right),
\end{equation}
%where, if the integrator satisfies $\Psi_h^{[p]}\circ\Psi_{-h}^{[p]}=\mathrm{Id}$, with $\mathrm{Id}$ denoting the identity flow, then it is a time-symmetric integrator. 
where if it is time-symmetric
\[
(a_s=0, \quad a_i=a_{s-i}, \quad b_i=b_{s+1-i}) 
\qquad \mbox{or} \qquad
(a_1=0, \quad a_{i+1}=a_{s+1-i}, \quad b_i=b_{s+1-i})
, 
\qquad i=1,2,\ldots
\]
%
% ($b_s=0, a_i=a_{s+1-i}, b_i=b_{s-i}, \ i=1,2,\ldots$) or ($a_1=0, a_{i+1}=a_{s+1-i}, b_i=b_{s+1-i}, \ i=1,2,\ldots$),
only odd powers of $h$ appears in its series expansion, as in (\ref{eq.Strang}). The scheme is usually referred as an $s$-stage method since it involves $s$ evaluations of the maps $ \e^{ha_{i} \hat A^{[1]}}$ and $\e^{hb_i \hat A^{[2]}}, \ i=1, \ldots,s $. However, if $b_s=0$ or $a_1=0$  the last map is not counted in the cost because it can be concatenated with the fist map in the following step due to the First Same As Last (FSAL) property, taking into account that $ \e^{ha_{i} \hat A^{[1]}} \e^{ha_{j} \hat A^{[1]}}= \e^{h(a_{i}+a_j) \hat A^{[2]}}$ and similarly for $\hat A^{[1]}$, and the scheme is referred as an $(s-1)$-stage method. 
  }

\subsection{Working in the duplicated phase space}
%\paragraph{Working in the extended phase space}

%To build high-order methods either by composition or by extrapolation it is convenient to start from a symmetric method, usually an explicit symmetric second order scheme. For this reason, many different schemes have been developed tailored for particular classes of problems in order to get explicit symmetric second order methods.
%In \cite{blanes19otc} it was shown that most of these method can be easily obtained using s simple splitting method on an extended phase space, an idea that has been used many times in the literature.
{
When the problem cannot be split into subproblems that can be integrated exactly,
one can instead duplicate the space and then split the problem.
To this end,}  
to solve \eqref{eq:ODE1} is consistent with the differential equation \cite{hairer93sod}
\begin{equation} \label{eq.D_Split}
	\begin{array}{l}
		u'= f(v), \quad \quad \\
		v'= f(u), \qquad 
	\end{array} \mbox{or \ equivalently} \quad \quad 	\frac{d}{dt} \left(
\begin{array}{c} u \\  v  \end{array} \right)=
\left(\begin{array}{c} 0 \\ f(u)  \end{array} \right)	
+
\left(\begin{array}{c} f(v) \\  0 \end{array} \right) 
=
f^{[1]}( u) + f^{[2]}(v),
\end{equation}
$u(t_0)=v(t_0)=x_0$,
%\begin{equation} \label{eq.3.StrangNonAut2}
%	\begin{array}{ll}
%		u'= f(v), \qquad &u(t_0)=x_0\\
%		v'= f(u), \qquad &v(t_0)=x_0.
%	\end{array}
%\end{equation}
%or, equivalently \cite{blanes19otc}
%\begin{equation}\label{eq.3.ODE}
%	\frac{d}{dt} \left(
%	\begin{array}{c} u \\  v  \end{array} \right)=
%	\left(\begin{array}{c} 0 \\ f(u)  \end{array} \right)	
%	+
%	\left(\begin{array}{c} f(v) \\  0 \end{array} \right) 
%	=
%	f^{[1]}( u) + f^{[2]}(v), \qquad
%	\left(
%	\begin{array}{c} u(t_0) \\  v(t_0) \end{array} \right)=	
%	\left(
%	\begin{array}{c} x_0  \\ x_0  \end{array} \right),
%\end{equation}
whose solution is $u(t)=v(t)=x(t)$.
In \cite{blanes19otc} it was shown that most explicit symmetric methods from the literature (used to build new high-order methods by composition or by extrapolation) can be obtained using a basic splitting method on this extended phase space. %, and that we can write as 
This system is separable into two trivially solvable autonomous parts and splitting methods can be applied. 
The scheme \eqref{eq.2.spt} applied to solve \eqref{eq.D_Split}  corresponds to the sequence ($u_0=v_0=x_n$)
{
\begin{equation*}\label{eq.2N_Strang}
	\begin{array}{rcl}
	v_{1/2}	&=&	v_0+\frac{h}2f(u_0),
	\\
	u_{1}	&=& u_0+hf(v_{1/2}),\\
	v_{1}	&=& v_{1/2}+\frac{h}2f(u_{1}).
\end{array}
\end{equation*}
}
%\blau
{Since the duplicated vector $(u_1,v_1)^\top$ constitutes a symmetric second-order numerical approximation, both $u_1$ and $v_1$ provide symmetric second-order approximations to the exact solution $x(t_n+h)$.}
%Both, $u_{1}$ and  $v_{1}$ are symmetric second order approximations to the solution, $x(t_n+h)$, and
An improved result is usually obtained with the average, $x_{n+1}=(u_{1}+v_{1})/2. $
%\[
%  x_{n+1}=\frac12(u_{1}+v_{1}).
%\]
Obviously, the algorithm to build a D-splitting method from an $s$-stage splitting method 
 is %($U=V=x_n$)
% \begin{equation}\label{eq:2N_Splitting}
% 	\begin{array}{rcl}
% 	v_{i}	&=&	v_{i-1}+ha_if(u_{i-1}),
% 		\\
% 	u_{i}	&=& u_{i-1}+hb_if(v_{i}), 
% 	\qquad	\qquad i=1,2,\ldots,s \\
% 	x_{n+1} &=& \frac12(u_{s}+v_{s}),
% 	\end{array}
% \end{equation}
 \begin{equation}\label{eq:2N_Splitting}
  \begin{minipage}[c]{0.4\textwidth}
\centering
\[ 
(D-splitting): \left\{ 	\begin{array}{rcl}
	    v_{i}	&=&	v_{i-1}+ha_if(u_{i-1}),
	 	\\
	 	u_{i}	&=& u_{i-1}+hb_if(v_{i}), 
	 \\ & & 	\quad	i=1,2,\ldots,s \\
	 	x_{n+1} &=& \frac12(u_{s}+v_{s}),
	 	\end{array}
	 	\right.
	 	\]
\end{minipage}
\hfill
  \begin{minipage}[c]{0.4\textwidth}
	\centering
%	(D-splitting)
	\begin{lstlisting}
		  $V$ := $x_n$
		  $U$ := $x_n$
		  for i=1:s do
		    $V$ := $V + ha_if(U)$
		    $U$ := $U + hb_if(V)$
		  end
		  $x_{n+1}$:=$\frac{1}{2}(U+V)$  
	\end{lstlisting}
\end{minipage}
\end{equation}      
which corresponds to the following explicit RK method ($u_0=v_0=x_n$)
%(the coefficients $c_i^{[a]}=\sum_{j=1}^{i}a_j,c_i^{[b]}=\sum_{j=1}^{i}b_j$ are, as we will see, for non-autonomous problems where the time is also considered as two different coordinates \cite{blanes19otc}) 
\begin{equation}    \label{RK-tablero}
%	\begin{array}{rcl}
%	v_{i}	&=&	v_{i-1}+ha_if(u_{i-1}),
%	\\
%	u_{i}	&=& u_{i-1}+hb_if(v_{i}), \\
%		& & 	i=1,2,\ldots,s \\
%	x_{n+1} &=& \frac12(u_{s}+v_{s}), \qquad
%\end{array} \quad
{
	\begin{array}{c|cccccccccc}
%		 &   &  &  &  &  &   &  &  &  &  \\
		c_1^{[a]} &   a_{1} &  &  &  &  &   &  &    &  & \\
		c_1^{[b]} &   0 &b_{1} &  &  &    &  &  &   &  & \\
c_2^{[a]} &   a_{1} & 0 &  a_{2} &  &   &   &  &   &  & \\
c_2^{[b]} &   0 &b_{1} &  0 & b_{2} &    &  &  &   &  & \\
c_3^{[a]} &   a_{1} & 0 &  a_{2} &   0 &  a_{3}    &  &  &  &  &  \\
c_3^{[b]} &   0 & b_{1} &  0 & b_{2} &  0 & b_{3}    &  &     &  &  \\
		\vdots & \vdots &  &           &   &  \ddots   &  & \ddots &  &  &  \\
		\vdots & \vdots &  &           &   &     &\ddots  &  & \ddots &  &%\text{\blau{llevem aquesta línia?}}  
		\\
c_s^{[a]}=1 | (v_s) &  a_1 &  0 & a_2 &   \ldots  &   &  &  a_{s-1}  &   0 & a_{s}  & \\
		\hline
c_s^{[b]}=1 | (u_s)	& 0 & b_1 &  0 & b_2 &  \ldots &     &  &b_{s-1}  & 0 &  b_s  
 \\
%\hline (v_s) &  a_1 &  0 & a_2 & 0 &   \ldots   &  &  a_{s-1}  &   0 & a_{s}  & 0 \\
\hline
 \frac12(u_s+v_s) &  \frac12 a_1 &  \frac12 b_1 & \frac12 a_2 &  \frac12 b_2  &  \ldots   &  &  \frac12 a_{s-1}  & \frac12 b_{s-1}   & \frac12 a_{s}  &  \frac12 b_s 		\end{array}
}
\end{equation}
 (the coefficients $c_i^{[a]}=\sum_{j=1}^{i}a_j,c_i^{[b]}=\sum_{j=1}^{i}b_j$ are, as we will see, required for non-autonomous problems where the time is also considered as two different coordinates \cite{blanes19otc}).
Given an splitting method of order $p$ then $u_s$ and $v_s$ are approximations to order $p$ and, obviously, this is also the case for the average, $(u_s+v_s)/2$. 
This scheme corresponds to an embedded $2s$-stage method of order $p$, i.e. an embedded  $(2s,p)$ RK method (if $a_1=0$ or $b_s=0$ the number of stages reduces to $2s-1$, instead of $2s-2$, because the FSAL property does not apply in this case).
% Once $u_s,v_s$ are obtained, one can compute $u_s:=u_s-v_s$ and to use $\|u_s\|$ as an estimation of the error to chose the following time step. If it is accepted, we take $x_{n+1}=v_s+\frac12 u_s$, otherwise if  $\|u_s\|$ is too large and the step is rejected, one can recover $x_n$ by first  evaluating $u_s:=u_s+v_s$ and then integrating backward with the inverse of the method (the same algorithm but in the oposite order and changing the sign of the coefficients). If rejections occur only occasionally, this is more economical than keeping a third storage register as in the 2R methods.

Since there are splitting methods to any order then they can be used as 2$N$-storage explicit RK method at any order. 
In addition, if the splitting method is symmetric, 
%i.e.
%\[
% (a_s=0, \quad a_i=a_{s-i}, \quad b_i=b_{s+1-i}) 
% \qquad \mbox{or} \qquad
% (a_1=0, \quad a_{i+1}=a_{s+1-i}, \quad b_i=b_{s+1-i})
% ,  \qquad i=1,2,\ldots
%\]
then $u_s$ and $v_s$ are symmetric schemes.
%, as well as $(u_s+v_s)/2$, and can also be used to obtain higher order methods by composition or extrapolation \cite{blanes19otc}.
Then, one can use the already existing optimised splitting methods from the literature \cite{mcLachlan02sm,yoshida90coh} (see also \cite{blanes24smf} and references there in).
% as well as methods using the processing technique \cite{blanes24foe,blanes99siw} (as far as the numerical solution is not requiered at each step).
%
%An $s$ stage splitting method of order $p$ requires then $2s$ evaluations of the vector field per step. On the other hand, two different solutions of order $p$ are obtained whose difference can be used as an error estimation. In addition, the average, in general, improves considerably the accuracy. 
At order $p=4$ there are methods with $s=4$ (3-stage symmetric compositions), but better results are obtained with the optimised method (BM4) 
%with $s=7$
%, and at order $p=6$ with $s=11$ 
\cite{blanes02psp}
{%\small
\begin{equation} \label{eq:coefsBM4}
\begin{array}{llll}
	a_1=0.07920369643119565, &
	a_2=0.353172906049774, &
	a_3=-0.04206508035771952, & a_4=1-2(a_1+a_2+a_3),\\
	b_1=0.209515106613362, &
	b_2=-0.143851773179818, &   b_3=\frac12-(b_1+b_2), 
\end{array}
\end{equation}
}
with $s=7, \ b_7=0, 
%a_4=1-2(a_1+a_2+a_3), \ b_3=\frac12-(b_1+b_2)
 \ a_{8-i}=a_i, b_{7-i}=b_i, \ i=1,2,3$, referred as a 6-stage symmetric splitting method \cite{blanes02psp} which, if used as a D-splitting method requires $2s-1=13$ evaluations per step, and the corresponding method is denoted as D$_{13}4$.%,  and
%   at order $p=6$
% {\small
% \[
% \begin{array}{llll}
% 	a_1=0.05026276440039223, &
% 	a_2=0.413514300428344, &
% 	a_3=0.04507988979439766, &
% 	a_4=-0.188054853819569, \\
% 	a_5=0.541960678450780, &  a_6=1-2(a_1+\ldots+a_5),
% 	&	& b_5=\frac12-(b_1+b_2+b_3+b_4), \\
% 	b_1=0.148816447901042, &
% 	b_2=-0.132385865767784, & 
% 	b_3=0.06730760469218501, & 	b_4=0.432666402578175, 
% \end{array}
% \]
% }
% with $s=11, \ b_{11}=0, \  \ a_{12-i}=a_i, b_{11-i}=b_i, \ i=1,2,3,4,5$, and $2s-1=21$ evaluations per step  (BM6)

\subsection{2$N$-storage splitting methods with a reduced number of stages}

While if $u_s$ and $v_s$ are of order $p$ implies that $(u_s+v_s)/2$ is also  of order $p$, the opposite is not necessarily true and it is possible that  $(u_s+v_s)/2$ is of higher order than $u_s$ and $v_s$. The simplest example is the Lie-Trotter method which corresponds to ($u_0=v_0=x_n$)
{
\begin{equation*}\label{eq:2N_Splitting}
	\begin{array}{rcl}
		v_{1}	&=&	v_{0}+hf(u_{0}),
		\\
		u_{1}	&=& u_{0}+hf(v_{1}),  \\
		x_{n+1} &=& \frac12(u_{1}+v_{1}) \ =x_n+\frac{h}2(f(x_n)+f(x_n+hf(x_n)))
	\end{array}
\end{equation*}
}
where both $u_1$ and $v_1$ are only first order approximations but the average is a well known second order RK method,
%\[
%x_{n+1}=x_n+\frac{h}2(f(x_n)+f(x_n+hf(x_n))),
%\] 
and this can also happen to higher orders.
The Lie operators associated to the vector fields $ f^{[1]}$ and $f^{[2]}$ are both closely related and some cancellations can occur when the average is taken, which reduces the number of order conditions with respect to standard splitting methods.
For example, given a splitting method of order $p$,
if
% with composition \eqref{eq.Splitting} then, if
{
\begin{equation}   \label{eq.symmetry}
\hat E_{p+1}({\hat{A}}^{[1]},\hat A^{[2]}) = 
-\hat E_{p+1}({\hat{A}}^{[2]},\hat A^{[1]})
\end{equation}
}
%\blau{
%\begin{equation}   \label{eq.symmetry}
%\hat E_{p+1}(\blau{\hat{A}}^{[1]},\hat A^{[2]}) = 
%\hat E_{p+1}(\blau{\hat{A}}^{[2]},\hat A^{[1]})+\mathcal{O}(h^{p+1})
%\end{equation}
%  }
the average cancels the contribution of the error at order $p$ and increases the order of the method in one unit. If the scheme is symmetric, this increment is of two orders. 
For instance, the composition
\begin{equation*}   \label{eq.Spl24}
%	\Psi_h^{[2]} =
	\e^{ha\hat A^{[1]}}
	\e^{h\frac12 \hat A^{[2]}} 
	\e^{h(1-2a) \hat A^{[1]}}
	\e^{h\frac12 \hat A^{[2]}} 
	\e^{ha\hat A^{[1]}}
	=
	\exp\left(h(\hat A^{[1]}+ \hat A^{[2]})+h^3\left(
\frac{6a^2-6a+1}{12}\hat E_{3,1}+\frac{6a-1}{24}\hat E_{3,2}	\right)+{\cal O}(h^5)
	\right)
\end{equation*}
with $\hat E_{3,1}=[\hat A^{[1]},[\hat A^{[1]},\hat A^{[2]}]], \hat E_{3,2}=[\hat A^{[2]},[\hat A^{[2]},\hat A^{[1]}]] $,
satisfies condition \eqref{eq.symmetry} for $p=2$ when the coefficient $a$ satisfies
%$\frac{6a^2-6a+1}{12}=-\frac{6a-1}{24} $, 
$\frac1{12}(6a^2-6a+1)=-\frac1{24}(6a-1) $, 
or
 $12a^2-6a+1=0$, i.e. if $a=\frac1{12}(3\pm i\sqrt{3})$ the scheme is a second order method which {turns} into a 4th-order method after the average is done. 
%
% The Lie algebra generated by the operators $\hat A^{[1]}$ and $\hat A^{[2]}$ is not a free Lie algebra and several cancellations can occur which deserves a further investigation. We have also observed that on Hamiltonian systems, the order of the pseudosymplecticity can be reduced when considering schemes where $u_s, v_s$ are of lower order than their average, and this also deserves further investigation and the results will be published elsewhere.
%
% As an illustration, we present a 4th-order symmetric splitting method with $s=7$ satisfying \eqref{eq.symmetry} for $p=4$ so, 
% after the average it turns into a 6th-order method (2N-S6) ($b_7=0, 
% %a_4=1-2(a_1+a_2+a_3), \ b_3=\frac12-(b_1+b_2)
% \ a_{8-i}=a_i, b_{7-i}=b_i, \ i=1,2,3$):
% {\small
% 	\[
% 	\begin{array}{llll}
% 		a_1=0.34117711626608893, &
% 		a_2=-0.11556397880852943, &
% 		a_3=0.0091007844006896624, & a_4=1-2(a_1+a_2+a_3),\\
% 		b_1=-0.19048598865349396, &
% 		b_2=-0.43215518907354579, &   b_3=\frac12-(b_1+b_2).
% 	\end{array}
% 	\]
% }
 %\blau
 {
 Furthermore, the associated Lie algebra framework allows us to isolate the error contributions of the linear part of the problem from those corresponding to the nonlinear component. This decoupling enables the construction of methods of order $p$ that achieve a higher order $q > p$ when applied to purely linear problems. As an illustrative example, we introduce a new method  with $s=7$ stages (matching the number of evaluations of the BM4 scheme). 
 
 In the composition \eqref{eq.Splitting} we have that the leading order term, $\hat E_{5}$, contains six indepednet elements of the Lie algebra that we can write as
 \[
  \hat E_{5} = \alpha_1[11112]+\alpha_2[21112]+\alpha_3[11212]+\alpha_4[22121]+\alpha_5[12221]+\alpha_6[22221]
 \]
 where $[11112]\equiv [\hat A^{[1]},[\hat A^{[1]},[\hat A^{[1]},[\hat A^{[1]},\hat A^{[2]}]]]]$, etc. It is easy to check that if the vector field, $f(x)$, is at most linear in $x$, i.e. $f(x)=Hx+b$ (there are tailored RK methods for this problem \cite{montijano23erk,zingg99rkm}), then
  \[
 [11112]=[21112]=[12221]=[22221]=0.
 \]

 In this design, the free parameters are exploited to attain $\alpha_3=\alpha_4=0$ which leads to sixth-order accuracy for linear problems, and the remaining free parameter is used to minimise the quantity, ${\cal E}_4=|\alpha_1|+|\alpha_2|+|\alpha_5|+|\alpha_6|$, in order to reduce the error when applied to non-linear problems, and we have chosen the following solution
{%\small
\begin{equation} \label{eq:coefsD13}
	\begin{array}{lll}
	a_1=0.06600649238911016, &
	a_2=-0.08393676964317165, &  \\
	a_3=0.35307044388088427, & a_4=1-2(a_1+a_2+a_3), & \\
	b_1=0.2992323094645197, &
	b_2=-0.0835865697375642, &   b_3=\frac12-(b_1+b_2).
\end{array}
\end{equation}
}
with $b_7 = 0,\ a_{8-i}=a_i,\ b_{7-i}=b_i,\  i =1,2,3$ which, as a D-splitting method requires $2s-1=13$ evaluations per step, and the corresponding method is denoted as D$_{13}4(6)$.
}

 The symmetry condition (\ref{eq.symmetry}) 
 corresponds to the following conditions
% is equivalent to the condition $\alpha_3=-\alpha_4$ suffices for a 6th-order method for linear problems after averaging, leaving 2 free parameters to reach even smaller values of the leading error ${\cal E}_4$. Alternatively, with the three free parameters it is possible go fulfill the conditions
 \[
\alpha_1= -\alpha_6, \qquad 
\alpha_2=-\alpha_5, \qquad 
\alpha_3=-\alpha_4 \qquad 
 \]
 which makes that, after the averaging, a 6th-order scheme is obtained. Since our composition has three free parameters, these conditions can be fulfilled. We have obtained an scheme to check that it corresponds to a method of order 6 after averaging. However, the best scheme we have obtained with this number of stages does not improve the performance already obtained with the D-splitting method based on the 10-stage 6th-order splitting method from \cite{blanes02psp}, and the search of new improved high order methods is under investigation. 

\subsection{D-splitting methods as pseudo-Lie-group RK methods}

It is well known that splitting methods preserve most qualitative properties of the exact solution if the system is split into parts which preserve those properties. In the following, we prove that for linear systems, D-splitting methods obtained from an splitting methods of order $p$ preserves the qualitative properties up to order $2p+1$ and they can be referred as a pseudo-geometric integrators. We prove this result for Lie-group methods.

Let us consider the linear equation
\begin{equation} \label{eq:linear}
 x'=Hx, \qquad x(0)=x_0\in\mathbb{C}^{N}
\end{equation}
where $H\in\mathbb{C}^{N\times N}$ is an element of a given Lie algebra. 
\begin{equation} \label{alg1}
	o_N(\mathbb{C},J) = \{ \Omega \in \mathbb{C}^{N \times N} \, : \,  \Omega^* J + J \Omega = 0 \},
%	o_J(N) = \{ \Omega \in \mathbb{C}^{N \times N} \, : \,  \Omega^* J + J \Omega = 0 \},
\end{equation}  
where $J$ is some constant matrix in $\mathrm{GL}_N(\mathbb{C})$, the group of all $N \times N$ non-singular complex matrices. In that case,
the fundamental matrix of \eqref{eq:linear}, $X(t)=\e^{tH}$, evolves in the Lie group
\begin{equation} \label{group1}
	\mathrm{O}_N(\mathbb{C},J) = \{ X \in \mathrm{GL}_N(\mathbb{C}) \, : \, X^* J X = J \},
%	O_J(N) = \{ X \in \mathrm{GL}_N(\mathbb{C}) \, : \, X^* J X = J \}
\end{equation}
whose associated Lie algebra is precisely $o_N(\mathbb{C},J)$. Relevant
examples of this class are the unitary and orthogonal group (when $J=I$, the identity matrix), the symplectic group (when $J$ is the basic
symplectic matrix) and the Lorentz group (corresponding to $J = \mathrm{diag}(1,-1,-1,-1)$).

Notice that 
\[
  H^*J=-JH\qquad \Rightarrow \qquad (H^{k})^*J=(-1)^kJH^{k}
\]
so, if $k$ is odd then $H^k$ is also an element of the Lie algebra. Then, if $F(x)$ is an odd analytical function in $x$, $F(-x)=-F(x)$, then $F(H)\in o_N(\mathbb{C},J)$.

\begin{theorem}
 An splitting method of order $p$ which is used as a D-splitting method in the duplicated space to solve eq. \eqref{eq:linear} is a pseudo-Lie group integrator which preserves the Lie group structure up to order $2p+1$.
\end{theorem}

\begin{proof}

%Notice that $\sigma(H)\in\mathbb{R}$ and the solution is given by the unitary transformation $u(t)=\e^{-itH}u_0$. 
Working in the duplicated space corresponds to solve the system

\begin{equation} \label{eq.3.StrangNonAut2}
	\frac{d}{dt} \left(
	\begin{array}{c} y \\  z  \end{array} \right)
	=
	\underbrace{\left(\begin{array}{cc} 0 & H \\ 0 & 0  \end{array} \right)}_{A} \left(	\begin{array}{c} y \\  z  \end{array} \right)
	+
	\underbrace{\left(\begin{array}{cc} 0 & 0 \\ H & 0 \end{array} \right)}_{B} \left(	\begin{array}{c} v \\  w  \end{array} \right).
\end{equation}
If an splitting method with real coefficients of order $p$ is used to solve this problem we have (using e.g. the BCH formula under convergence conditions)
\begin{equation*}   \label{eq.Prod}
	\Psi_h^{[p]} =
	\e^{hb_s B} \e^{ha_s A}
	\cdots
	\e^{hb_1 B} \e^{ha_1A}
	= 
	\exp\left(\begin{array}{cc} a(hH) & b(hH)\\ c(hH) & -a(hH)  \end{array} \right),
\end{equation*}
where  $a$ is an even real function of $H$ and $b,c$ are odd real functions of $H$. For a method of order $p$ we have that
\[
a(x)={\cal O}(x^{p+1}), \qquad
b(x)=x+{\cal O}(x^{p+1}), \qquad c(x)=x+{\cal O}(x^{p+1})
\]
so, we can write
\[
a(hH)=hH\alpha(hH)), \qquad b(hH)=hH(I+\beta(hH)), \qquad c=hH(I+\gamma(hH))
\]
with $\alpha(hH)={\cal O}(h^{p})$ being an odd function of $hH$ and $\beta(hH)={\cal O}(h^{p}), \gamma(hH)={\cal O}(h^{p})$  both being even functions of $hH$. 
Let us now introduce the function 
\[
\eta=(bc+a^2)^{1/2}
=hH\Big((I+\beta)(I+\gamma)+\alpha^2\Big)^{1/2}
=hH\Big(I+\beta+\gamma+\beta\gamma+\alpha^2\Big)^{1/2}
\]
which is an odd function of $hH$ since $\beta+\gamma+\beta\gamma+\alpha^2$ is an even function of $hH$. In addition, taking into account the Taylor expansion of the square root function we have that
\begin{equation}   \label{eq.eta}
\eta=hH\Big(I+\frac12(\beta+ \gamma)+\frac12(\beta\gamma+\alpha^2)+{\cal O}(h^{2p})\Big)=\frac12(b+c)+{\cal O}(h^{2p+1}).
\end{equation}
On the other hand, we have that\footnote{For simplicity in the presentation we write $\frac{a}{\eta}$ instead of $a\eta^{-1}$ or $\eta^{-1}a$ since $a$ and $\eta$ are functions of $hH$ which commute.}
\begin{equation*}   \label{eq.Exp}
	\Psi_h^{[p]} =
	\exp\left(\begin{array}{cc} a & b\\ c & -a  \end{array} \right)
	= \left(\begin{array}{cc} \cosh(\eta)+\frac{a}{\eta}\sinh(\eta) & 
		\frac{b}{\eta}\sinh(\eta)\\ \frac{c}{\eta}\sinh(\eta) & \cosh(\eta)-\frac{a}{\eta}\sinh(\eta)  \end{array} \right).
\end{equation*}
 Notice that $\eta=hH+{\cal O}(h^{p+1})$ and it is an odd function of $H$ so, it is an element of the Lie algebra.
Then, one step of a D-splitting method corresponds to
\begin{equation*}   \label{eq.Exp2}
\left(
\begin{array}{c} y_s \\  z_s  \end{array} \right)
	= \left(\begin{array}{cc} \cosh(\eta)+\frac{a}{\eta}\sinh(\eta) & 
		\frac{b}{\eta}\sinh(\eta)\\ \frac{c}{\eta}\sinh(\eta) & \cosh(\eta)-\frac{a}{\eta}\sinh(\eta)  \end{array} \right)
		\left(
		\begin{array}{c} x_n \\  x_n  \end{array} \right),
\end{equation*}
and, taking into account \eqref{eq.eta}
%\[
%\eta=((hH+\beta h^{p+1})(hH+\gamma h^{p+1})-\alpha^2 h^{2p+2})^{1/2}=hH\Big(I+\frac12(\tilde \beta+\tilde \gamma)h^{p+1}+{\cal O}(h^{2p+1})\Big)=\frac12(b+c)+{\cal O}(h^{2p+1})
%\]
%where $\beta=h\tilde \beta, \gamma= h\tilde \gamma$, 
%we have that
\[
x_{n+1}=\frac12(y_s+z_s)=(\cosh(\eta)+\frac{b+c}{\eta}\sinh(\eta))x_n		
= \e^{\eta}x_n +{\cal O}(h^{2p+2}),
\]
so, $x_{n+1}$ is an approximation of order $p$ (since $\e^{\eta}=\e^{hH}+{\cal O}(h^{p+1})$) that preserves the Lie group structure up to order $2p+1$ since $\e^{\eta}\in \mathrm{O}_N(\mathbb{C},J)$.

\end{proof}

{\sc Remark:} This result shows that, in general, the order at which a D-splitting method preserves the Lie-group structure depends on the order, $p$, of the splitting method, but not on the order reached after the average, which can be higher than $p$.

\section{Numerical examples}

We analyse the performance of D-splitting methods built both from already existing splitting methods as well as the new tailored method on some linear and non-linear ODEs and PDEs.

The following methods are considered for the numerical examples:
\begin{itemize}
	\item  RK$_4$4: The standard 4-stage fourth order RK method (as a standard reference method).
	\item RK$_7$6: The 7-stage sixth order RK method with coefficients given in \cite[p. 194]{butcher08nmf} (used as a reference standard method).
	\item  RK$^{[2R]}_5$4: The 5-stage fourth order 2R-storage RK method from \cite{kennedy00lse} denoted as RK4(3)5[2R+]C.
%	\item RK$^{[2R]}_9$5: The 9-stage fifth order 2R-storage RK method from \cite{kennedy00lse} denoted as RK5(4)9[2R+]S.	
%	\item D$_5$4: The 5-stage 4th order D-splitting methods \eqref (obtained from a 2-stage second order splitting methods with complex coefficients)
	\item  D$_{13}$4: The 13-stage 4th order D-splitting method  obtained from the 6-stage 4th order splitting methods with coefficients given in \eqref{eq:coefsBM4}.
	\item  D$_{13}$4(6): The new 13-stage 4th order D-splitting method with coefficients given in \eqref{eq:coefsD13},  which is a 6th-order order method for linear problems. 
\end{itemize}

\paragraph{1D wave equation}
A simple problem to test D-splitting methods as 2$N$-storage methods is the following linear hyperbolic equation with periodic boundary conditions \cite{carpenter94fo2}:
{\small
\begin{align*}
  \frac{\partial u}{\partial t}+\frac{\partial u}{\partial x} = 0,\quad &0\le x \le 1,\ t\ge 0,\\
  u(0,t)=u(1,t)=-\sin(8\pi t),\quad &t\ge 0,\\
  u(x,0)=\sin(8\pi x),\quad &0\le x \le 1,
\end{align*}
}
for which the exact solution is known to be $ u(x,t)=\sin(8\pi(x-t)). $
%\begin{equation*}
%  u(x,t)=\sin(2\pi(x-t)),\quad 0\le x \le 1,\ t\ge 0.
%\end{equation*}
To integrate this problem numerically, we begin by partitioning the spatial interval into
$N$ subintervals of length $\Delta x=1/N$.
The numerical solution is then represented by the vector
$\tilde{u}=(u_0,\dots,u_{N-1})$, where $u_i\simeq u(x_i,t)$ with $x_i=i\Delta x$.
If a Fourier spectral collocation method is used, one obtains a
system of $N$ ODEs of the form 
$\tilde{u}'=-A\tilde{u}=f(\tilde{u}), $ 
%
%\begin{equation*}
%  \frac{d}{dt}\tilde{u}=-A\tilde{u}=f(\tilde{u}),
%\end{equation*}
where $A$ is a differentiation matrix.
Its action can be computed as
$A\tilde{u}=\mathcal{F}^{-1}D_A\mathcal{F}\tilde{u}$,
where $\mathcal{F}$ and $\mathcal{F}^{-1}$
denote the forward and inverse discrete Fourier transforms, respectively, and
$D_A$ is a diagonal matrix \cite{trefethen00smi}.
%
% To perform the simulations, we consider the following methods:
% %. On the one hand, we use classical Runge–Kutta (RK) methods: 
% the two-stage second order Heun's RK method (RK2), the well-known four-stage fourth-order RK method (RK4), the five-stage 2$N$-storage method (KCL4) taken from \cite{kennedy00lse} (denoted as RK4(3)5[2R+]C), the Strang splitting method (\ref{eq.Strang}) (S2), the previous splitting methods (BM4) and (BM6) from \cite{blanes02psp}, 
% %of fourth and sixth order, respectively, 
% and our new sixth-order method (2N-S6). When implemented as 2$N$-storage methods, the splitting schemes require 3, 13, 21, and 13 stages, respectively.

%\begin{figure}[!h]
%  \begin{center}
%    \subfloat[]{
%      \label{fig:ona-1}
%      \includegraphics[width=7cm]{./img/error_ona.pdf}}
%    \subfloat[]{
%      \label{fig:ona-2}
%      \includegraphics[width=7cm]{./img/error_unitarietat.pdf}}
%    \caption{Error in the solution (a) and in the unitarity (b) for the one-dimensional wave equation problem with final time $t_f = 50$ and %$N = 32$.}        \label{fig:ona}
%  \end{center}
%\end{figure}
%
\begin{figure}[!h]
	\centering
	\subfloat{
		\includegraphics[width=5.5cm]{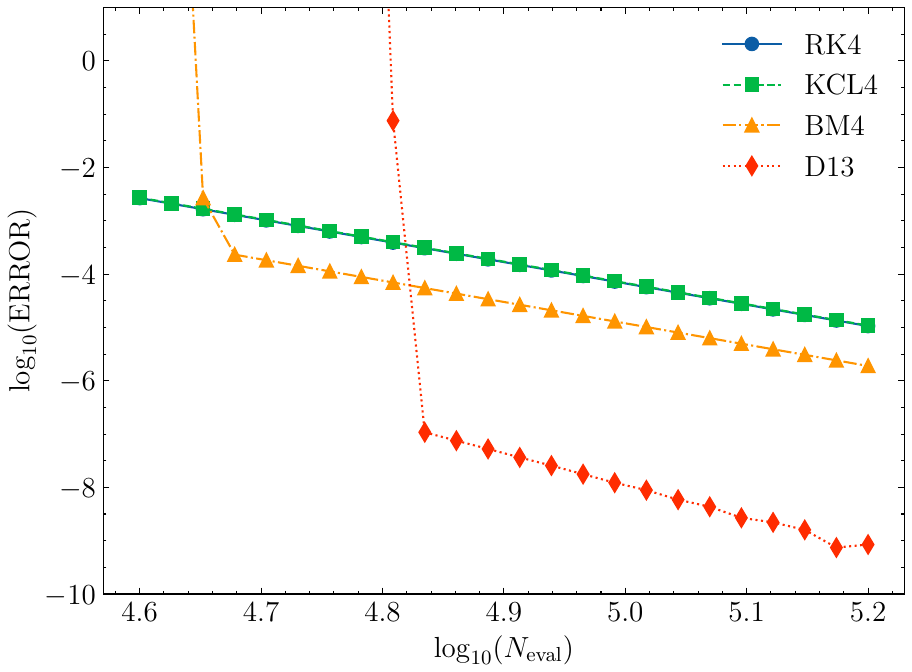}}
	\subfloat{
		\includegraphics[width=5.5cm]{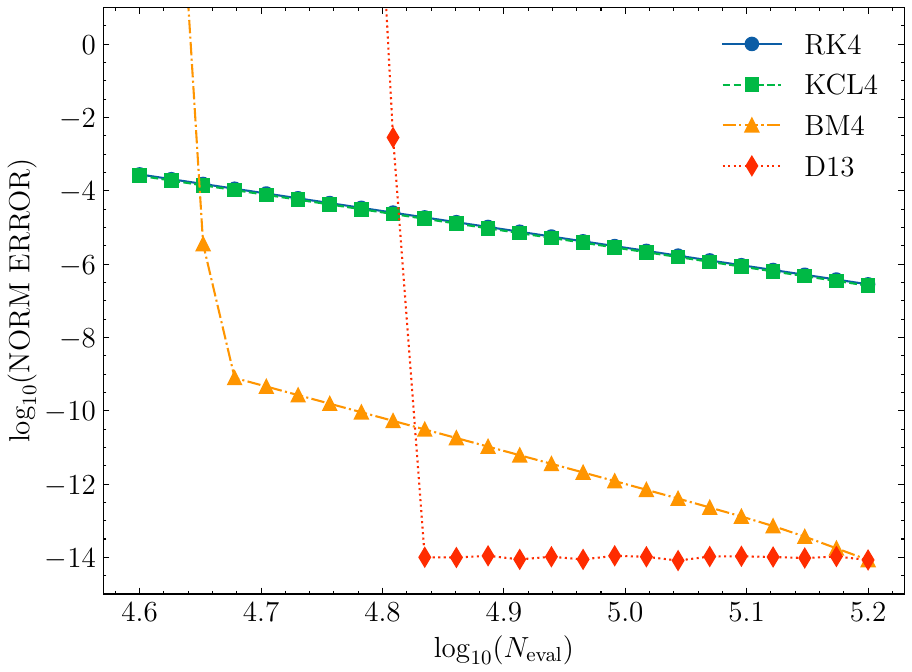}}
	\caption{Error in the solution (left) and in the mass (right) for the one-dimensional wave equation problem at final time $t_f = 50$ and $N = 128$.}    
	\label{fig:ona}
\end{figure}
The integrations are performed up to the final time $t_f = 50$, as in \cite{kennedy00lse}, using a spatial discretization with $N =128$ grid points. Figure~\ref{fig:ona} shows the results obtained where both plots display, on a logarithmic scale, the error versus computational cost, measured in terms of the number of function evaluations.
The left panel shows the error in the wave function with respect to the exact solution, while the right panel depicts the error in the conservation of mass:
% the wave function norm:
{\small
\begin{equation*}
  \text{ERROR: }\lVert\tilde{u}(t_f)-u(t_f)\rVert /\lVert u(t_f)\rVert, \qquad 
  \text{NORM ERROR: }
  (|\lVert\tilde{u}(t_f)\rVert -\lVert\tilde{u}(t_0)\rVert |)/\lVert\tilde{u}(t_0)\rVert,
%  (|\lVert\tilde{u}(t_f)\rVert_{L^2} -\lVert\tilde{u}(t_0)\rVert_{L^2} |)/\lVert\tilde{u}(t_0)\rVert_{L^2},
\end{equation*}
}
where $\lVert u\rVert=\sqrt{u^{\top}u}$.
% and $\lVert u\rVert_{L^2}=\Delta x\sum_{i=0}^{N-1}u_i^2$.
Some methods exhibit a more restrictive stability range; in this regard, it is worth emphasizing the favorable stability properties of BM4 (D$_{13}4$).
Nevertheless, within their respective stability regions, the D-splitting methods demonstrate improved efficiency compared to classical
RK schemes. In particular, the observed slopes for the norm conservation error indicate orders higher than the formal order of the methods.
For the D$_{13}4(6)$ method, whenever stability is maintained, the error reaches machine precision.

Next, in order to illustrate the performance of the new methods on linear problems which evolve on Lie-groups as well as to observe the order of accuracy in the preservation of the qualitative properties, we consider the numerical integration of the linear problem
\[
X'=HX, \qquad \qquad \qquad X(0)=I_N
\] 
where $I_N$ is the identity matrix of dimension $N$.
We take $N=100$, integrate until $t_f=10$ and measure the error in the solution at the final time, $X(t_f)=\e^{t_f H}$ versus the computational cost measured as the number of matrix-matrix products for different choices of the time step. The error in the solution and in the preservation of the Lie-group structure are measured as follows
\begin{equation*}
	%	\text{ERROR: }
	{\cal E}_1= \lVert\tilde{X}(t_f)-X(t_f)\rVert /\lVert X(t_f)\rVert, \qquad  \qquad  \qquad 
	%	\text{SYMPL. ERROR: }
	{\cal E}_2= 
	|\lVert\tilde{X}(t_f)^TJ\tilde{X}(t_f) -J \rVert. %-\lVert\tilde{u}(t_0)\rVert |)/\lVert\tilde{u}(t_0)\rVert,
	%  (|\lVert\tilde{u}(t_f)\rVert_{L^2} -\lVert\tilde{u}(t_0)\rVert_{L^2} |)/\lVert\tilde{u}(t_0)\rVert_{L^2},
\end{equation*}
Notice that for, this problem, the storage requirement corresponds to store full matrices, which is equivalent to work on a problem with vectors of dimension $N^2=10^4$. In particular, we consider the following symplectic and the unitary problems:

\begin{figure}[!h]
	\begin{center}
		\includegraphics[height=15cm]{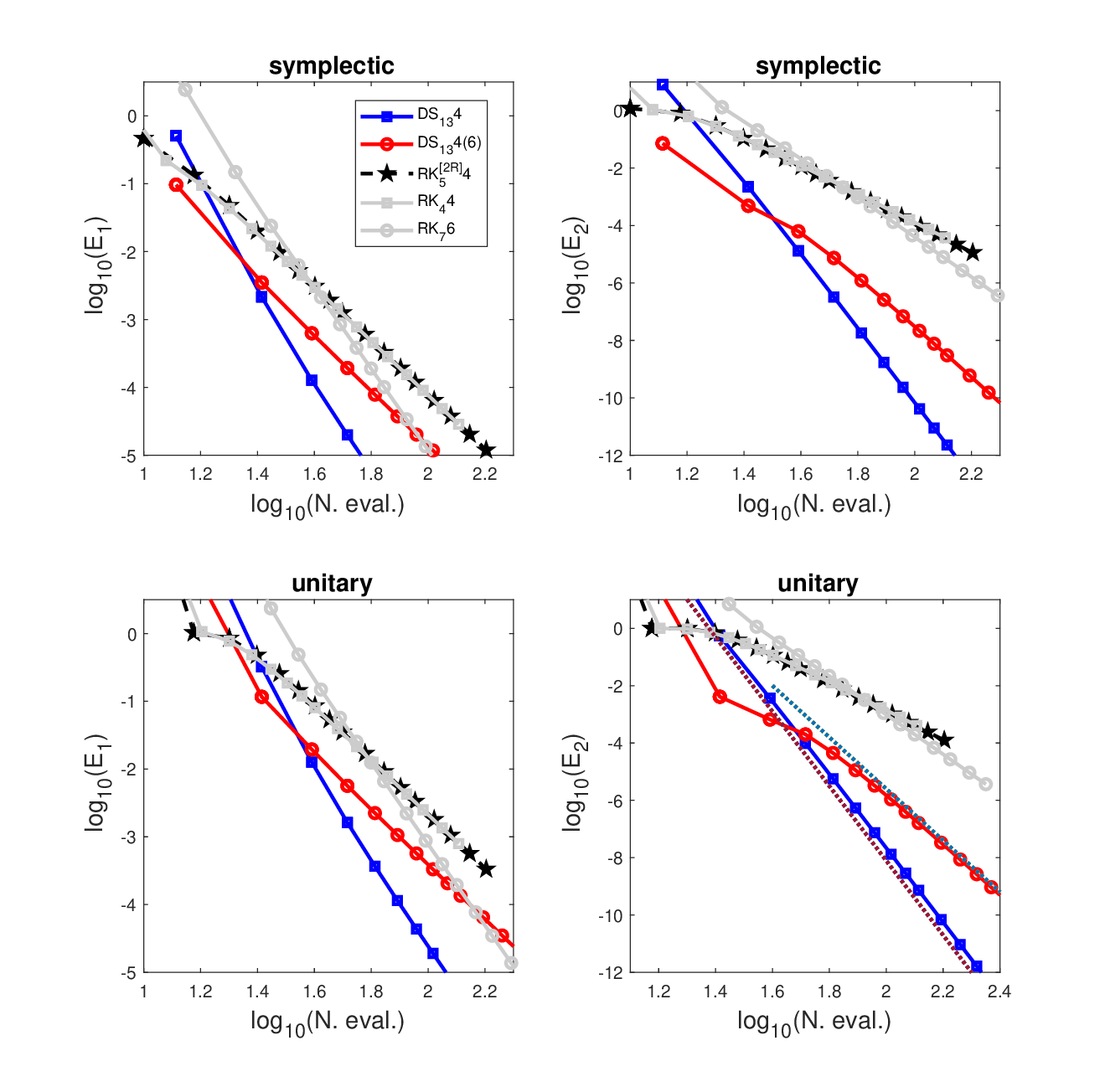}
		\caption{Error in the matrix solution, ${\cal E}_1$, at the final time $t_f=10$, $\e^{t_fH}$ (left panels) and error in the preservation of the Lie-group structure at the final time, ${\cal E}_2$, (right panels) for the symplectic and unitary examples where the matrix $H$ is given by \eqref{H:symplectic} and \eqref{H:unitary}, respectively, versus the computational cost measured as the number of matrix-matrix products. Dotted lines in the bottom right panel are reference curves with slopes corresponding to orders 9 and 13.}
		\label{fig:SymplecticUnitary}
	\end{center}
\end{figure}

\paragraph{Linear symplectic problem}
We consider the case in which
\begin{equation} \label{H:symplectic}
H=JS, \qquad  \qquad  \qquad J=\left(\begin{array}{cc} 0_d & I_d\\ -I_d & 0_d  \end{array} \right)
\end{equation}
$N=2d$, $0_d,I_d$ are the zero and identity matrices of dimension $d$, and $S\in\mathbb{R}^{N\times N}$ is a symmetric positive definite matrix which we build as follows 
\[
\displaystyle S=  \frac1{\|V^TV\|}V^TV
\] 
and $V$ a random matrix obtained in {\sc Matlab} with seed $\mathtt{rng(1)}$ and  $\mathtt{V=randn(d)}$. 
Fig~\ref{fig:SymplecticUnitary}, top figures show the results obtained. We observe that the new D-splitting methods outperform the already existing 2R-storage method (with slightly reduced stability) and are also superior to the standard RK methods. In the right panel we show the error in the preservation of symplecticity where the superiority of D-splitting methods much clear and the slopes of the curves are approximately as a method of order 9 for the scheme D$_{13}4$ and of order 13  for the scheme D$^{[6]}_{13}4(6)$ which agree with the expected results.

\paragraph{Linear unitary problem}
We consider now the case in which
\begin{equation} \label{H:unitary}
H=\frac1{\|A+iB\|}(A+iB), \qquad  \qquad  \qquad J=I_N
\end{equation}
where $A$ is a random real skew-symmetric matrix and $B$ is a random real symmetric matrix obtained as follows
\[
\mathtt{rng(1)};  \qquad \mathtt{A=randn(N)};  \qquad  \mathtt{A=A'-A};  \qquad  \mathtt{B=randn(N)};   \qquad \mathtt{B=B'+B};
\]

Fig~\ref{fig:SymplecticUnitary}, bottom figures show the results obtained, where we observe similar results as for the symplectic case. In the bottom right panel we include two reference dotted lines of slope corresponding to schems of order 9 and 13, in agreement to the preservation of unitarity by D-splitting methods built from splitting methods of order 4 and 6, respectively.
%We observe that the new D-splitting methods outperform the already existing 2R-storage method (with slightly worst stability) and are also superior to the standar RK methods. In the right panel we show the error in the preservation of symplecticity where the superiority of D-splitting methods much clear and the slopes of the curves are approximately as a method of order 9 for the scheme DS$_{13}4$ and of order 13  for the scheme DS$_{13}4(6)$ wich agree with the expected results.

%The Taylor method with the Horner's algorithm can be used as a 2$N$-storage method, being one of the most efficient methods in the literature to solve this problem since it attains order $m$ with only $m$ products while storing only two vectors in memory \cite{almohy11ctaf}. 
%On the other hand, in \cite{} [Gray-Manolopoulos] the authors presented a set of $m$-stage splitting methods of order $m$ for $m=2,4,6,8,10$ and 12, for the linear Schr\"odinger, which can also be used as D-splitting methods on the duplicated space at a cost of $2m$ products per step.. 

We now analyse the performance of D-splitting methods on  some non-linear ODEs and PDEs and their preservation of the qualitative properties.

\paragraph{Two body Kepler problem}
Although the Kepler problem can be split into two integrable parts, we consider here its classical formulation in order
to apply splitting methods in a $2N$-storage framework:
\begin{equation*}
 (q',p')^\top=\left(p, -q/r^3\right)^\top,\text{with } r=\lVert q\rVert, \quad q=(q_1,q_2), p=(p_1,p_2), \text{and } q_i,p_i\in\mathbb{R},
\end{equation*}
with initial conditions
$  q_1(0)=1-e, q_2(0)=p_1(0)=0, \text{and } p_2(0)=\sqrt{(1+e)/(1-e)},
$
%\begin{equation*}
%  q_1(0)=1-e, q_2(0)=p_1(0)=0, \text{and } p_2(0)=\sqrt{(1+e)/(1-e)},
%\end{equation*}
which yield an elliptical trajectory of period $2\pi$,  eccentricity $e$ and energy $E_0 = -1/2$.
We take  $e = 0.8$ and integrate until the final time $t_f = 2000$.
For the numerical integration, we employ the two fourth-order RK methods introduced previously,
together with the D-splitting methods D$_{13}4$ and  D$_{13}4(6)$.
The step sizes are chosen so that all methods require 520000 force evaluations.
%have the same computational cost.
Figure~\ref{fig:kep} shows the error in energy along the propagation for these three methods.
%The simulations are performed with eccentricity $e = 0.8$ and final time $t_f = 2000$.
The results show that the RK methods exhibit a linear growth of the energy error from the outset,
whereas both D-splitting methods initially preserve the energy error before transitioning to a linear growth regime.
This behavior is characteristic of pseudo-symplectic methods.

\begin{figure}[!h]
  \begin{center}
    \includegraphics[height=5.5cm]{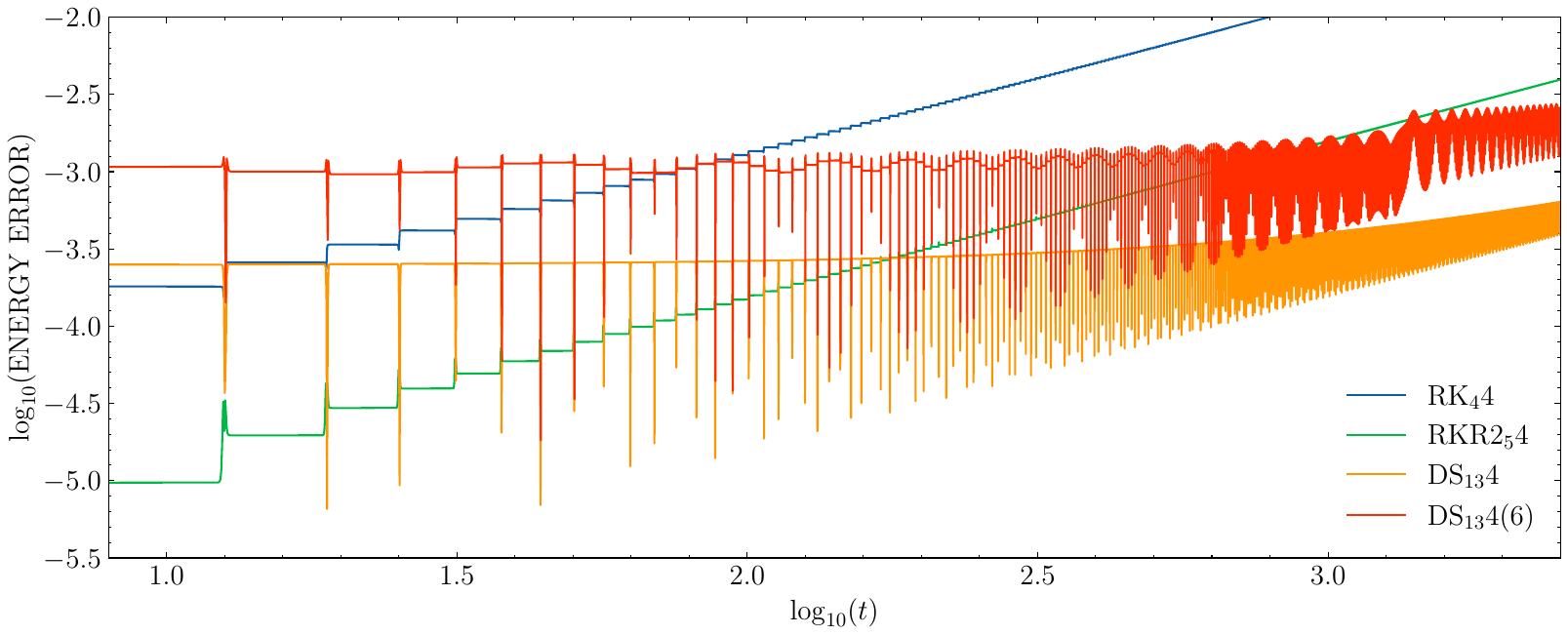}
    \caption{Propagation of the energy error in the Kepler problem with $e = 0.8$ and $t_f = 2000$ for three integrators, the third exhibiting pseudo-symplectic behavior. All three simulations were performed under an identical computational cost.}
    \label{fig:kep}
  \end{center}
\end{figure}

% \paragraph{Example 1} 
% The wave equation
% \paragraph{Example 2} Some Hamiltonian systems.

% \begin{figure}
% 	\centering
% 	%		\input{figures/ex1.tiz}
% %	\includegraphics[width=\textwidth]{Fig1_RatDecomp.eps}
% 	\caption{}
% 	\label{fig:}
% \end{figure}

\paragraph{The Gross-Pitaevskii equation} As a final example we consider the one-dimensional Gross-Pitaevskii equation, which is a non-linear Schr\"odinger equation given by

\begin{equation} \label{Schr1}
	i \varepsilon  \frac{\partial}{\partial t} \psi (x,t) = -\frac{\varepsilon^2}{2 \mu} \frac{\partial^2 \psi}{\partial x^2}(x,t) + \hat{V}(x) \psi(x,t) + \kappa_2|\psi(x,t)|^2 \psi(x,t), \qquad x \in [-L,L],
\end{equation}
Following the numerical example presented in \cite{bao03nso}, we take periodic boundary conditions in $x \in [-L,L],$ i.e. $\psi(-L,t) = \psi(L,t)$ for all $t$ with $L=16$, $\mu=1$, $\varepsilon=0.1$, $V(x)=\frac12 x^2$ and $\kappa_2=1.2649$ with initial conditions $\psi (x,0)=\rho\e^{-x^2/(2\varepsilon)}$ and $\rho$ a normalizing constant. The spatial interval is divided into $N=512$ intervals of length $\Delta x=2L/N$ leading to a system of non-linear ODEs
\[
  {\bf u}'={\bf f}({\bf u}), \qquad {\bf u}(0)={\bf u}_0
\] 
where ${\bf u}_i(t)\simeq \sqrt{\Delta x}\ \psi (x_i,t)$ and $\|{\bf u}(0) \|=1$. 
 We integrate along the time interval $t\in[0,2]$ and measure the error in the wave function, ${\cal E}_1$ (the exact solution of the spatially discretised problem is computed numerically to sufficiently high accuracy) and in the norm, ${\cal E}_2$. The action of the Laplacian is computed with FFTs and we measure the error versus the cost given by the required number of FFTs.

We have repeated the numerical experiments for $\kappa_2=0.12649/50$ (a weaker contribution from the non-linear part) and initial conditions $\psi (x,0)=\rho\e^{-(x-4)^2/(2\varepsilon)}$ with $\rho$ a normalizing constant.
% to avoid to stay too close to the ground state. 
Figure~\ref{fig:SchrodGP} shows the results obtained. We observe that the new D-splitting methods provide in this case more accurate results and behave as higher order methods in the preservation of unitarity, similarly to the linear case.

\begin{figure}[!h]
	\begin{center}
		\includegraphics[height=13cm]{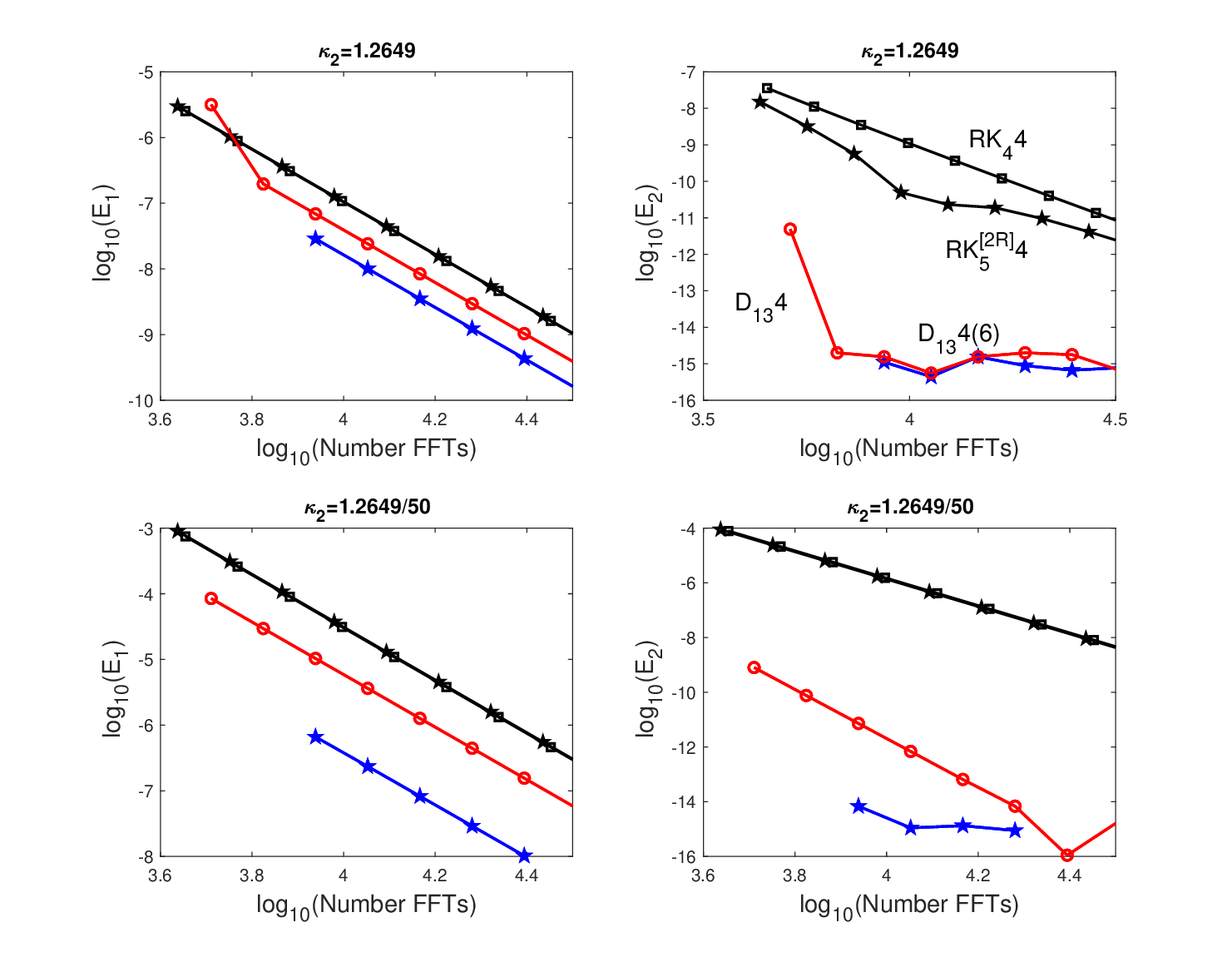}
		\caption{Error in the solution of the semidiscretised Gross-Pitaevskii equation
			at the final time $t_f=2$ (left panels) and in the preservation of the norm (right panels) versus the computational cost measured as the number of FFTs. The value of the constant, $\kappa_2$, for the non-linear term is reduced in the bottom figures in order to make the system as near to linear.}
		\label{fig:SchrodGP}
	\end{center}
\end{figure}

\section{Conclusions}
%\paragraph{Conclusions}
Beyond problems that can be split into several parts,
  splitting methods can also be applied to 
  %\blau
  {more general} problems.
  %\blau
  {In particular, they allow, in the duplicated phase space,} the construction of 2$N$-storage embedded explicit RK methods
  of arbitrarily high order while remaining very easy to implement.
  Currently, the literature contains a large number of splitting
  and composition methods that can be applied as 2$N$-storage RK methods.
%
%Splitting methodsallow to build 2$N$-storage RK methods at any order, and are \blau{very} simple to implement.
%
%\blau{
  Furthermore, the results presented in this work open several promising research directions, among which the following deserve special attention:
  \begin{itemize}
  \item We have also shown that it is possible to build new D-splitting methods that, after the averaging, allow to reach a higher order than the order of the splitting method. This requires a deeper analysis in the order conditions based on the algebraic structure of the Lie operators in the duplicated phase space, and this study is under investigation.
    We have also observed that on Hamiltonian systems, the order of the pseudo-symplecticity can be reduced when considering schemes where $u_s, v_s$ are of lower order than their average, and this also deserves further investigation and the results will be published elsewhere.
    \item One can also build methods of order $p$ for general non-linear problems which, for linear problems, are of a higher order, $q>p$, with improved stability,
    % Splitting methods on the extended phase space can 
or one can use the processing technique \cite{blanes24foe,blanes99siw} (if the numerical solution is not required at each step).
They can also be considered as pseudo-geometric 2$N$-storage methods when the vector field has a particular algebraic structure, and new pseudo-symplectic and pseudo-unitary methods will be built 
% and analysed for the numerical integration of classical and quantum mechanic problems 
and published elsewhere.
\item D-splitting methods can be used as embedded schemes in a variable time-step setting. Once $u_s,v_s$ are obtained, one can compute $u_s:=u_s-v_s$ and to use $\|u_s\|$ as an estimation of the error to chose the following time step. If it is accepted, we take $x_{n+1}=v_s+\frac12 u_s$. Otherwise, if  $\|u_s\|$ is too large and the step is rejected, one can recover $x_n$ by first  evaluating $u_s:=u_s+v_s$ and then integrating backward with the inverse of the method (the same algorithm but in the opposite order and changing the sign of the coefficients). If rejections occur only occasionally, this is more economical than keeping a third storage register at each step.
Obviously, an understimation of the error can occur when $\|u_s-v_s\|$ is smaller than the true error. It is possible to obtain a lower order estimator from the intermediate results in the splitting method with a third storage register \cite{blanes19sac}, 
 as in the 2R methods. %This feature may also be exploited to construct adaptive time-stepping 2$N$-storage methods.
\end{itemize}
}
Finally, the results from this letter 
%\blau
{can be} extend to the non-autonomous case, $x'=f(t,x)$, as follows
\begin{equation*} \label{eq.3.StrangNonAut2}
	\begin{array}{ll}
		u'= f(v_t,v), \quad u_t'=1, \qquad\qquad &u(t_0)=x_0, \ u_t(0)=t_0\\
		v'= f(u_t,u), \quad v_t'=1, \qquad\qquad &v(t_0)=x_0, \ v_t(0)=t_0,
	\end{array}
\end{equation*}
i.e. taking the time as two new coordinates \cite{blanes19otc} and using the coefficients $c_i^{[a]},c_i^{[b]}$ from the Butcher tableaux \eqref{RK-tablero}.
% Then, the D-splitting methods can also be used as pseudo-geometric methods to numerically solve the linear non-autonomous problem
%\[
%x'=H(t)x,
%\]
%where, in the case $H(t)$ evolves smoothly on time, the problem can be considered as nearly linear so, tailored methods of order $p$ which are of order $q>p$ for the linear case could be of great interest.

%There exist highly efficient splitting methods which are optimised for the separable linear problem
%\begin{equation} \label{eq.SyplecticSplit}
%	\frac{d}{dt} \left(
%	\begin{array}{c} q \\  p  \end{array} \right)
%	=
%	\underbrace{\left(\begin{array}{cc} 0 & M\\ 0 & 0  \end{array} \right)}_{A} \left(	\begin{array}{c} q \\  p  \end{array} \right)
%	+
%	\underbrace{\left(\begin{array}{cc} 0 & 0 \\ N & 0 \end{array} \right)}_{B} \left(	\begin{array}{c} q \\  p  \end{array} \right).
%\end{equation}
%for the case in which $N=-M$ and $N$ is a symmetric real matrix where upper and lower bounds to the eigenvalues are known. These methods can also be used as $D$-splitting methods if one takes $M=N=H$ and we will analyse their performance and stability as well as if it is possible to build new improved methods.

\section*{Acknowledgements}
%\paragraph{Acknowledgements}
The work of both authors is supported by Ministerio de Ciencia e Innovación (Spain) through project PID2022-
136585NB-C21, MCIN/AEI/10.13039/501100011033/FEDER, UE.  AE-T acknowledge financial support by Universitat Jaume I
through project UJI-2025-06.

\end{document}